\documentclass[reqno,12pt]{amsart}
\usepackage{amsmath,amsthm,amssymb,amsfonts,epsfig,color,graphicx,enumerate,accents,subfigure,comment}
\usepackage[a4paper,margin=2.59truecm]{geometry}

 \usepackage{comment}

\DeclareMathOperator{\dive}{div}

\DeclareMathOperator{\dist}{dist}

\def\eps{{\varepsilon}}
\def\N{\mathbb{N}}
\def\R{\mathbb{R}}
\def\O{\Omega}
\def\Ob{\overline{\Omega}}
\def\A{\mathcal{A}}

\def\F{\mathcal{F}}
\def\HH{\mathcal{H}}

\newcommand{\be}{\begin{equation}}
\newcommand{\ee}{\end{equation}}

\numberwithin{equation}{section}
\theoremstyle{plain}
\newtheorem{theo}{Theorem}[section]
\newtheorem{lemm}[theo]{Lemma}

\theoremstyle{remark}
\newtheorem{rema}[theo]{Remark}

\title{A non-isotropic free transmission problem governed by quasi-linear operators}

\author{Harish Shrivastava}
\address{Departamento de Matemática, Universidade Federal da Paraíba(Jo\~ao Pessoa, Brazil)}
\email{h.shrivastava@studenti.unipi.it}
\begin{document}

\begin{abstract}
We study a free transmission problem in which solution minimizes a functional with different definitions in positive and negative phases. We prove some asymptotic regularity results when the jumps of the diffusion coefficients get smaller along the free boundary. At last, we prove a measure-theoretic result related to the free boundary.
\end{abstract}
\medskip

\maketitle


\textbf{2010 Mathematics Subjects Classification:}49J05, 35B65, 35Q92, 35Q35

\tableofcontents

\section{Introduction}\label{introduction}

In this article we intend to study the regularity issues related to the transmission problems. In various applied sciences, many phenomenas are modelled by transmission problem also known as phase transition problems. These kind of models naturally appear when we study the diffusion of a quantity through different media. For example, modelling a composite material having different diffusion properties: like combination of ice and water, or mixture of chemicals, a tumour in some tissue, or heat conduction through different regions. 

Very broadly speaking, variational formulation of a transmission problem is of the following form
$$
\int_{\O}a_+(x,v,\nabla v)\chi_{\{v>0\}}+a_-(x,v,\nabla v)\chi_{\{v\le 0\}}\,dx \;\to \,\min
$$ 
for an appropriate domain $\O\subset \R^N$, $a_+$ and $a_-$ determine diffusion in positive and negative phases. Candidates $v$ are from appropriate function space. The ice-water example is the most relatable because the (solid) ice part corresponds to the negative phase and (liquid) water part corresponds to the positive phase.

The mathematical analysis of transmission problems involves discontinuous coefficients, due to the difference in the properties of different media. Let us focus on the stationary state of the ice-water combination and study the diffuson of heat (related to the temperature $T$ ) $T:\O\to \R^N$, $\O$ being the domain under study. We can say that in ice the diffusion is determined by an operator corresponding to solid state of water that is
$$
-\dive(a_-(x)\nabla T)=0\qquad\mbox {in ice, $\{T<0\}$} 
$$
and in water, the diffusion is determined according to an operator corresponging to liquid state
$$
-\dive((a_+(x)\nabla T)=0\qquad\mbox{in water, $\{T\ge0\}$}.
$$
As a comination of the above two PDEs we can write 
$$
-\dive(a(x)\nabla T)=0\qquad\mbox{in $\O$}
$$
with $a(x,T)=a_+(x)\chi_{\{T>0\}}+a_-(x)\chi_{\{T\le0\}}$. With $a$ being a discontinuous function along the free boundary of $T$. An important point to be noticed is that the diffusion tend to compensate the transition of phases, which gets reflected in the free boundary condition. It can be formally written as follows, supposing that the free boundary is sufficiently regular
$$
\mathcal {G}(\partial _{\nu_+}T,\partial _{\nu_-}T)=0\qquad\mbox{on $\partial \{T>0\}$}
$$
for some function $\mathcal G$. Above mentioned PDEs and the free boundary condition can be posed in the following variational setup, which was studied in \cite{TA15} 
\be\label{TA15}
\int_{\O}\langle A(x,v)\nabla v ,\nabla v \rangle-f(x,v)v+\gamma(x,v)\,dx 
\ee
with 
\[
\begin{split}
A(x,v)=A_+(x)\chi_{\{v>0\}}+A_-(x)\chi_{\{v\le 0\}}\\
f(x,v)=f_+(x)\chi_{\{v>0\}}+f_-(x)\chi_{\{v\le 0\}}\\
\gamma(x,v)=\gamma_+(x)\chi_{\{v>0\}}+\gamma_-(x)\chi_{\{v\le 0\}}
\end{split}
\]

and matrices $A_{\pm}$ satisfying the ellipticity condition for any $\xi\in \R^N$
$$
\lambda |\xi|^2 \le \langle A_{\pm}\xi,\xi \rangle \le \Lambda |\xi|^2,
$$
$f_{\pm}\in L^N(\O)$, $\gamma_{\pm}\in C(\Ob)$. 

One important point to note is that the functionals involved in the phase transition problems are not convex, hence the existence result does not follow from classical methods. The approach involve some tools from measure theory and also variational calculus (see \cite{LTQ15}, \cite{JDS18}, \cite{TA15}).

We remark that the addition of the last term $\gamma(x,u)$ (commonly called the compensation term) penalizes the change of phases, which in turn imposes some regularity on the free boundary. The role of this term is very evident in the Section \ref{section5} where we prove that the last term forces the free boundary to be a rectifiable set of finite $\mathcal H^{N-1}$ measure. The technique to show rectifiability of the free boundary is adapted from \cite{db12}, also see \cite{ButHar18} for an application of the same technique in shape optimization. We expect the free boundary to be even more regular, and the compensation term should play an important role in it.

In \cite{TA15} authors have shown that as the discontinuity of the diffusion coefficients gets smaller the solution $u$ of \eqref{TA15} gets more regular, tending to Lipschitz regularity. We can imagine it as studying the behaviour of diffusion when the material becomes more and more homogenized with time.

In this article, we have considered a functional corresponding to a quasi-linear operator in respective phases (see \eqref{P}). The problem of phase transitions can be seen as a generalization of the free boundary problems studied by Alt, Caffarelli and Friedman (see \cite{altcaf81} for one phase problem and \cite{ACF84} for two phase model). One can see results in \cite {ACF84} as a particular case of the functional \eqref{TA15} with $A_+=A_-=Id$ and $f_+=f_-=0$. 

In fact, we can see the variational problem dealt in this article as a combination of the problems which fall into into two broad categories, Bernoulli type free boundary problems (with the source term $f=0$, see \cite{altcaf81}, \cite{ACF84} for the linear case, and \cite{dp05} for a non-linear one phase problem) and obstacle type problem (with the compensation term $\gamma=0$ see \cite{FKR17}, \cite{FGS17} ). Very roughly speaking, the minimizers of Bernoulli type functionals are less regular (at most Lipschitz continuous) while solutions of an obstacle problem can carry up to $C^{1,1}$ regularity. Since we are studying a mixture of both of the above mentioned problems, it is reasonable to think that one should not expect a minimizer to be more than Lipschitz continuous. The observations made in \cite{TA15}, \cite{LTQ15} indicate the same. One can refer to \cite{KLS17} where Alt-Caffarelli-Friedman monotonicity formula with two different operators is established and used to show the Lipschitz continuity for solution of PDEs with jump discontinuity in the operator.

Another prominent work related to free transmission problems can be found in \cite{JDS18}, dealing with the functionals of the form \eqref{JDS18} such that a solution satisfy PDEs with different non-linearities in different phases.
\be\label{JDS18}
\int_{\O\cap \{u>0\}}|\nabla u|^p-f_+(x)u+\gamma_+(x)\,dx+\int_{\O\cap \{u\le 0\}}|\nabla u|^q-f_-(x)u+\gamma_-(x)\,dx.
\ee
Also refer to \cite{LTQ15} where the above functional with $p=q$ is studied, proving that the solutions are locally log-Lipschitz in the domain. The functional \eqref{JDS18} is also studied in \cite{AF94}, assuming the free boundary is a fixed surface with Lipschitz regularity.

We present the mathematical setup which we will be working on in this article. $\O \subset \R^N$ is open, smooth and bounded. $N\ge 3$, $\A_+, \A_- \in L^{\infty} (\O)$ and satisfy the following boundedness condition
\be\label{ellipticity}
\lambda  \le  \A_{\pm}(x) \le \Lambda 
\ee 
for almost every $x\in \O$, $0<\lambda\le\Lambda<\infty$ are fixed constants. $\gamma_{\pm}$ are continuous and integrable real valued functions  on $\O$, $p\in [2,N)$ is fixed. $f_{\pm}\in L^q(\O)$ for $q>\frac{N}{p}$. 

We consider a functional of the form
\be\tag{$\mathcal P$}\label{P}
\F_{\A,f,\gamma}(v;\O )=\int_{\O} \A (x,v)  |\nabla v|^{p}  -f(x,v)v+\gamma(x,v)\,dx
\ee
where the integrand is defined as 
\[
\begin{split}
\A (x,s)&= \A_+(x)\chi_{\{s>0\}}+\A_-(x)\chi_{\{s\le 0\}}\\
f(x,s)&= f_+(x)\chi_{\{s>0\}}+f_-(x)\chi_{\{s\le 0\}}\\
\gamma(x,s)&= \gamma_+(x)\chi_{\{s>0\}}+\gamma_-(x)\chi_{\{s\le 0\}}.
\end{split}
\]
For candidate functions in the search of a minimizer, we consider the following Sobolev space with fixed boundary data $\phi\in W^{1,p}(\O)$
$$
W^{1,p}_{\phi}(\O)=\left \{ v\in W^{1,p}(\O) \lvert v-\phi\in W_0^{1,p}(\O) \right \}.
$$

In the absence of any ambiguity, we will denote $\F_{\A,f,\gamma}$ solely as $\F$, and we will mention only the subscripts which carry a risk of being ambiguous. Note that any quantity depending solely on $\A,f,\gamma,\phi,\O$ will be referred to as quantity depending on data of the problem.

As mentioned earlier, the functional $\F$ in \eqref{P} is not convex in $W^{1,p}(\O)$, we will prove that $\F$ is lower semi-continuous with respect to $v$ in $W^{1,p}(\O)$ topology via techniques from measure theory (see Theorem \ref{existence}). In the Theorem \ref{holder}, we use the results from the theory developed by Giaquinta-Giusti (see \cite{gigi84}, \cite{eg05}) which is related to quasi minima of a functional and conclude local boundedness and existence of a universal modulus of continuity for all the minimizers of $\F$. Method used in the Theorem \ref {holder} differ from the one used in \cite{TA15}, we believe Giaquinta-Giusti's arguments used in this article can be applied to more general classes of transmission problems. 

The universal modulus of continuity plays a decisive role in giving compactness arguments while studying the asymptotic regularity of solutions in the Section \ref{section3} and \ref{section4}. 

In the Section \ref{section3} we will show that if $\A_{+}=\A_-=\A\in C(\O)$ and $f_{\pm}\in L^{N}(\O)$, then $u\in C^{0,1^-}_{loc}(\O)$, using tangential analysis method.  Main idea is to study regularity of solution of $\F$ with the coefficients $\A_{\pm}$ such that $\F$ is close to a given tangential free boundary problem. The arguments in the Section \ref{section3} and \ref{section4} can also be posed in terms of $\Gamma$-convergence (see \cite{dg68}, \cite{ab02}, \cite{eg05} for a comprehensive introduction to the subject), but we have refrained from using this term in the proofs. In the Section \ref{section4}, we use analogous arguments as in the Section \ref{section3} to show that as the jumps between $\A_{\pm}\in C(\O)$ gets smaller, solutions tend to be more regular, asymptotically tending to Lipschitz regularity.

In the last section, Section \ref{section5}, we prove that the free boundary $\partial ^*\{u>0\}$ of a minimizer $u$ of $\F$ in \eqref{P} is always a set of finite perimeter under the assumption of Dirichlet boundary condition i.e. $\phi=0$ and $(\gamma_+-\gamma_-)>c>0$. We can also prove similar result for a general boundary condition, but in order to avoid tedious calculations which will digress the reader from the main idea of the proof, we have chosen to provide the proof for only Dirichlet boundary, and key steps for the general case is mentioned in the Remark \ref{general boundary}.

We also remark that the assumption on the ordering on $\gamma_{\pm}$ i.e. $\gamma_+>\gamma_-$ can be dropped and replaced with $\gamma(x,s)>0$ for $s\neq 0$, $\gamma(x,0)=0$. In this case we can prove rectifiability (finite perimeter) of  a larger set $\partial ^*\{|u|>0\}$

The proof in the last section involves techniques from geometric measure theory, we refer the reader to \cite{fmgmt}, \cite{HF69}, \cite{AFP00},  \cite{PK08}, \cite{GE15} for definitions and preliminary results used to prove the Theorem \ref{finite perimeter}.

\section{Existence and minimal H\"older regularity}
We combine the methods in calculus of variations and measure theory to show the existence of a minimizer, note that the functional is not convex (see the discussion in \cite{TA15} for a counter example).  An approach similar to ours can be found in \cite{LTQ15}.

\begin{theo}\label{existence}
Given a boundary data $\phi\in W^{1,p}(\O)$, there exists a minimizer $u \in W^{1,p}_{\phi}(\O)$ of $\F$ in the problem \eqref{P} .
\end{theo}

\begin{proof}
Since $f\in L^q(\O)$, $q>\frac{N}{p}$ and $\A_{\pm}$ satisfy the boundedness condition \eqref{ellipticity} we have
\be\label{1.1}
\F(v)\ge \int_{\O}\lambda|\nabla v|^p-f(x,v)v+\gamma(x,v)\,dx
\ee
and since $p<N$ and $f\in L^q(\O)$ for $q>\frac{N}{p}>{p^*}'$, therefore by H\"older inequality and Poincaré inequality, we have
\be\label{1.2}
\int_{\O}f(x,v)v\,dx\le C(N)\|f\|_{L^{{p^*}'}(\O)}\|v\|_{L^{p^*}(\O)}\le C(N,q) \|f\|_{L^q(\O)}(\|\nabla v\|_{L^p(\O)} +C(\phi))
\ee
since $\gamma_{\pm}(x)$ are integrable in $\O$, the last term $\int_{\O}\gamma(x,v)\,dx$ is bounded. Combining this fact with \eqref{1.1} and \eqref{1.2}, we have
\be\label{1.3}
\F(v)\ge \lambda \|\nabla v\|_{L^p(\O)}^p -C(N,q)\|f\|_{L^q}\left (\|\nabla v\|_{L^p(\O)}+C(\phi) \right )+C(\gamma)>-\infty
\ee
for all $v\in W^{1,p}_{\phi}(\O)$. Thus, we establish existence of a lower bound for the functional $\F$. As there exists a minimum value, let $\{u_n\}$ be a minimizing sequence in $W^{1,p}_{\phi}(\O)$, by standard arguments (use Poincar\'e inequality on $\{(u_n-\phi)\}$, $u_n\in W^{1,p}_{\phi}(\O)$ ), we can show that 
$$
\sup_{n\in \N}\F(u_n)<\infty \Rightarrow \sup_{n\in \N}\|u_n\|_{W^{1,p}(\O)}<\infty.
$$
Hence $u_n$ is a bounded sequence in $W^{1,p}(\O)$ norm, by reflexivity of $W^{1,p}(\O)$, the sequence $u_n$ has a weak limit upto a subsequence. 

Since $\O$ is a bounded set therefore by Rellich theorem (see \cite{evans}, \cite{HB10}), $W^{1,p}(\O)$ embeds compactly into $L^p(\O)$. Therefore, there exists a function $u_0\in W_{\phi}^{1,p}(\O)$ such that $\nabla u_n \rightharpoonup \nabla u_0$ in $L^p(\O)$ and $u_n\to u_0$ in $L^p(\O)$ upto a subsequence. Moreover, we know that $u_n\to u_0$ pointwise almost everywhere in $\O$ upto another subsequence. By Egorov's theorem (\cite{GBF99}, \cite{WR87}), given an $\eps>0$ there exists a set $\O_{\eps}\subset \O$ such that $|\O\setminus \O_{\eps}|<\eps$ and $u_n \to u_0$ uniformly in $\O_{\eps}$. Fix $\delta >0$ and we see that 
\be
\begin{split}\label{lsc1}
\int_{\O_{\eps}\cap \{u_0>\delta\}}\A(x,u_0) |\nabla u_0|^{p}\,dx&=\int_{\O_{\eps}\cap \{u_0>\delta\}}\A_{+}(x)|\nabla u_0|^p\,dx\\
&\le \liminf_{n\to \infty}\int_{\O_{\eps}\cap \{u_0>\delta\}}\A_+(x) |\nabla u_n|^{p}\,dx\\
&\le\liminf_{n\to \infty} \int_{\O_{\eps}\cap \{u_n>\frac{\delta}{2}\}} \A_+(x) |\nabla u_n|^{p}\,dx\\
&\le \liminf_{n\to \infty}\int_{\O\cap \{u_n>0\}} \A_+(x) |\nabla u_n|^{p} \,dx\\
&= \liminf_{n\to \infty} \int_{\O\cap \{u_n>0\}}  \A(x,u_n) |\nabla u_n|^{p} \,dx
\end{split}
\ee
and from \eqref{ellipticity} we can write
\be\label{1.5}
\int_{\O\setminus \O_{\eps}} \lambda |\nabla u_0|^p\,dx \le \int_{\O\setminus \O_{\eps}}  \A(x,u_0) |\nabla u_0|^{p}\le \int_{\O\setminus \O_{\eps}} \Lambda |\nabla u_0|^p \,dx\to 0 \mbox{  as $\eps\to 0$}.
\ee
By letting $\delta\to 0$ and $\eps\to 0$, combine \eqref{lsc1} and \eqref{1.5} and we have
\be\label{lsc+}
\int_{\O\cap \{u_0>0\}}  \A(x,u_0)|\nabla u_0|^{p}\,dx\le \liminf_{n\to \infty} \int_{\O\cap \{u_n>0\}} \A(x,u_n) |\nabla u_n|^{p} \,dx.
\ee
By considering the set $\O_{\eps}\cap \{u_0<-\delta\}$ in the equations \eqref{lsc1}, we can argue analogously to say that
\be\label{lsc-}
\int_{\O\cap \{u_0\le0\}}\ \A(x,u_0)|\nabla u_0|^{p} \,dx\le \liminf_{n\to \infty} \int_{\O\cap \{u_n\le0\}}  \A(x,u_n) |\nabla u_n|^{p}\,dx
\ee
lower semi-continuity of the other terms in $\F$ is well known, that is
$$
\int_{\O}f(x,u_0)u_0\,dx \,\le\,  \liminf_{n\to \infty} \int_{\O}f(x,u_n)u_n\,dx
$$
and since $u_n\to u_0$ pointwise almost everywhere in $\O$,
$$
\int_{\O}\gamma(x,u_0)\,dx \, \le \, \liminf_{n\to \infty}\int_{\O}\gamma(x,u_n)\,dx.
$$
Along with \eqref{lsc+} and \eqref{lsc-}, we get
$$
\F(u_0)\le \liminf_{n\to \infty}\F(u_n)=\min
$$
and it follows that $u_0$ is a minimizer of $\F$ in $W^{1,p}_{\phi}(\O)$. This concludes the Theorem \ref{existence}.

\end{proof}

Now that we have established the existence of a minimizer of $\F$, we can proceed to prove that any minimizer of $\F$ posses a mimimal local H\"older continuity in the domain $\O$. For the sake of convinience in notation, we define the following for $x\in \O$, $s\in \R$, $\xi \in \R^N$,
$$
F(x,s,\xi)=  \A(x,s)  |\xi|^{p} -f(x,s)s+\gamma(x,s) 
$$
and observe that there exists a $C>0$ such that $s\le s^p+C$ for all $s\ge 0$, using this and \eqref{ellipticity} we can write that
\be\label{estimates}
\lambda|\nabla u |^p-|f||u|^p-(C|f|+|\gamma|)\le F(x,u,\nabla u)\le \Lambda |\nabla u |^p+|f||u|^p+(C|f|+|\gamma|).
\ee
where, by slightly abusing the notation, we define $|f|=|f_{+}|+|f_-|$ and $|\gamma|=|\gamma_+|+|\gamma_-|$.

\begin{rema}
Existence result holds true for more general values of $p\in (1, \infty)$. Since further regularity results are known only for the range of $p$ considered above, we choose to stick to the limit $p\in [2,N)$. 
\end{rema}

\begin{theo}\label{holder}
Given a minimizer $u\in W^{1,p}(\O)$ of $\F$, then $u$ is locally bounded and locally H\"older continuous in $\O$. That is for all $\O'\Subset \O$ there exists a $M(\O')>0$ and $0<\alpha_0<1$ depending only on data of the problem such that
$$
\|u\|_{C^{\alpha_0}(\O')}\le M\|u\|_{L^{\infty}(\O')}.
$$
\end{theo}
\begin{proof}
Since, $F$ satisfies the estimates \eqref{estimates}, the minimization problem \eqref{P} falls into the general setting of the variational problems studied in \cite{eg05}. That is, it satisfies the condition (7.2) in \cite[Section 7.1]{eg05}.

Note that, $F$ satisfies the hypothesis of \cite[Theorem 7.3]{eg05} which proves the local boundedness of $u$ in terms of its $L^{p}$ norm. From local boundedness of $u$ and \cite[Theorem 7.6]{eg05} we have local H\"older regularity of a minimizer of $\F$.
\end{proof}

\section{Small jumps and regularity in continuous medium}\label{section3}
In this section we will be proving that when $\A_+=A_-=\A\in C(\O)$ and $f_{\pm}\in L^N(\O)$, a minimizer $u$ of $\F$ satisfy local $C^{0,1^-}$ regularity estimates. We interpret the functional studied in \cite{LTQ15} as a tangential  free boundary problem, this strategy is adapted from \cite{TA15}. One can see \cite{TU13}, \cite{TE12}, \cite{TE11} for other applications of similar strategy as ours.

Also, we shall be proving the theorems in this section for a unit ball with centre at the origin. Which, on rescaling will represent a small ball contained inside a general domain $\O$. As we have already established the local boundedness of a solution in any general domain in the previous section, we will be assuming that for a minimizer $u$ of $\F$ in $B_1=B_1(0)$, $\|u\|_{L^{\infty}(B_1)}=1$.

We will first prove that as the oscillation of the diffusion coefficients $\A_{\pm}$ gets smaller, the graph of a minimizer $u$ of $\F$ tends to the graph of a $C^{0,1^-}$ function in $B_{1/2}$. This will lead us to asymptotic $C^{0,1^-}$ estimates in $B_{1/2}$ on the points located on the free boundary. 
Then we will use the Moser-Harnack inequality and some geometric arguments to prove that $u$ is locally $C^{0,1^-}$ regular when $\A_+=\A_-=\A\in C(\O)$.

\begin{lemm}\label{convergence}
Under fixed boundary condition $\varphi$, let $u\in W_{\varphi}^{1,p}(B_1)$ be a minimizer of functional $\F(\cdot,B_1)$ with $f_{\pm}\in L^N(B_1)$ then, for every $\eps>0$ there exists a $\delta>0$ such that if 
\be\label{coefficients}
\|\A-\A_0\|_{L^1(B_1)}\le \delta
\ee
for some constant $\A_0$, $\lambda \le \A_0 \le \Lambda$, then there exists a function $u_0\in C^{0,1^-}(B_{1/2})$ such that 
\be\label{uniform limit}
\|u-u_0\|_{L^{\infty}(B_{1/2})}\le \eps.
\ee
$u_0$ is such that for every $0<\beta<1$ the constant of  $\beta$-H\"older continuity, $C_0(\beta)$ depends only on $\A_0,\beta$ and data of the problem.

\end{lemm}
\begin{proof}
We argue by contradiction, i.e. let there is a sequence of functions $\A_k$ satisfying \eqref{ellipticity} such that
\be\label{converging coeff}
\lim_{k\to \infty} \|\A_k-\A_0\|_{L^1(B_1)}=0
\ee 
and there exists $\eps_0>0$ such that
for every $w\in C^{0,1^-}(B_{1/2})$ 
\be\label{contradiction}
\|u-w\|_{L^{\infty}(B_{1/2})}>\eps_0.
\ee

Before moving into the main body of the proof, we stop to make some observations. From the hypothesis \eqref{converging coeff}, we write $\A_{\pm,k}\to \A_{0}$ in $L^1(B_1)$. Then, upto a subsequence, $\A_{\pm,k}\to \A_0$ pointwise almost everywhere in $B_1$. 

The functionals $\F_k=\F_{\A_k}$ satisfy the hypothesis of the Theorem \ref{holder} in $B_1$, thus we can show that the functions in the sequence $\{u_k\}$ of minimizers of $\F_k$ are locally H\"older continuous in $B_1$. That is
$$
\|u_k\|_{C^{\alpha_0}(\overline B_{1/2})}<K_0
$$
for some $K_0>0$, not depending on $k\in \N$. $\alpha_0$ is as in the Theorem \ref{holder}.

Therefore, we can apply Arzela Ascoli theorem to $\{u_k\}$ and  there exists a function $h\in C^{0,\alpha_0}(B_{1/2})$ such that 
\be\label{AA limit}
u_k\to h\;\;\mbox{uniformly in $B_{1/2}$ up to a subsequence}.
\ee 

Therefore by compact embedding (see \cite{evans}, \cite{HB10}), we have  a function $u_0\in W_{\varphi}^{1,p}(B_1)$ such that 
\be\label{weak limit}
\begin{split}
\{u_k\}\rightharpoonup u_0 &\mbox{ in $W^{1,p}(B_1)$ and}\\
\{u_k\}\to u_0& \mbox{ in $L^p(B_1)$}
\end{split}
\ee
up to a subsequence.
From \eqref{AA limit} and \eqref{weak limit} we can say that $u_0=h$ almost everywhere in $B_{1/2}$.  Also note that $\|u_k\|_{W^{1,p}(B_1)}$ is uniformly bounded in $k$, from the assumptions mentioned in the beginning of the section, \eqref{ellipticity} and minimility of $u_k$ for the functional $\F_k$ we have

\be\label{equiintegrability}
 \begin{split}
\lambda \int_{B_1}|\nabla u_k|^p\,dx \,\le \int_{B_1} \A_k(x,u_k)|\nabla u_k|^p\,dx\le \F_k(\phi)+\int_{B_1}f(x,u_k)u_k-\gamma(x,u_k)\,dx\\
\, \le \Lambda \int_{B_1}|\nabla \phi|^p\,dx + C<C_0<\infty.
\end{split} 
\ee

We proceed to show that the function $u_0\in W_{\varphi}^{1,p}(B_1)$ is a minimizer of 
\be\label{F0}
\F_0(v)=\int_{B_1}\A_0 |\nabla v|^{p}-f(x,v)v+\gamma(x,v)\,dx
\ee
observe that
\be\label{principle part}
\liminf_{k\to \infty} \int_{B_1} \A_k (x,u_k) |\nabla u_k|^{p} \,dx  \ge   \int_{B_1} \A_0 |\nabla u_0|^{p}\,dx.
\ee
Indeed the inequation \eqref{principle part} can be shown via a set of arguments similar to those in the proof of the Theorem \ref{existence}. Only little difference is that we need to apply Egorov's theorem to the sequences $\{u_k\}$ as well as $\{\A_k\}$. Fix $\eps'>0$ and let $\O_{\eps'}\subset B_1$ be such that $\A_k \to \A_0$ and $u_k \to u_0$ uniformly in $\O_{\eps'}$ and $|B_1\setminus \O_{\eps'}|<\eps'$. From this information, one can easily verify that the sequence $\{ |\A_k(\cdot,u_k)|^{\frac{1}{p}}\nabla u_k \}$ weakly converges to $|\A_0|^{\frac{1}{p}} \nabla u_0$ in $L^{p}(\O_{\eps'})$.

Then we can proceed exactly like \eqref{lsc1} - \eqref{lsc-} to prove \eqref{principle part}.

As $k\to \infty$ from \eqref{AA limit} we have
\be\label{FB part}
\liminf_{k\to \infty}\int_{B_1} \gamma(u_k,x)\,dx  \ge \int_{B_1}\gamma(u_0,x)\,dx.
\ee
Adding \eqref{principle part} and \eqref{FB part}, we get the following inequality:
\be\label{part1}
\liminf_{k\to \infty} \F_k(u_k) \ge\,\F_0(u_0)
\ee
Moreover, for any $v\in W_{\varphi}^{1,p}(B_1)$
\be\label{part2}
\F_0(v) \ge  \F_k(v) 
+ \int_{B_1 } (\A_{0}-\A_{k}(x,v))  |\nabla v|^{p} \,dx
\ee
the last term in \eqref{part2} goes to Zero (following the same reasoning as \eqref{principle part}) as $k\to \infty$. Therefore from \eqref{part1} and \eqref{part2}, we have
$$
\F_0(v)  \ge \lim_{k\to \infty} \F_k(v)\ge \liminf_{k\to \infty}\F_k(u_k)\ge \F_0(u_0) 
$$
this shows that $u_0$ is the minimizer of $\F_0$.

From \cite{LTQ15} we know that the function $u_0$ is a locally log-Lipschitz (and therefore locally $C^{0,1^-}$) function in $B_{1}$, that is $u_0\in C^{0,1^-}(B_{1/2})$. If we take $w=u_0$ in \eqref{contradiction}, we get a contradiction. Hence we prove the lemma.

\end{proof}

\begin{lemm}\label {far from FB 0}
With $0\in \partial \{u>0\}$ and hypothesis of the previous lemma, for every $0<\alpha<1$ there exists a $0<r_0 <1/2$ and a $\delta=\delta(\alpha)>0$ such that if 
$$
\|\A-\A_0\|_{L^1(B_1)} <\delta
$$
then
$$
\sup_{B_{r_0}}|u|\le r_0^{\alpha}.
$$
\end{lemm}

\begin{proof}
For an $\eps>0$ to be chosen later (and accordingly $\delta>0$) we have from the Lemma \ref{convergence} that 
\be\label{5.1}
\|\A-\A_{0}\|_{L^{1}(B_1)}<\delta \Rightarrow \|u-u_0\|_{L^{\infty}(B_{1/2})}<\eps
\ee
for some $u_0\in C^{0,1^-}(B_{1/2})$, that is, for every $0<\beta<1$ there exists $C(\beta)>0$ such that 
\be\label{5.2}
\sup_{B_r}|u_0(x)-u_0(0)|\le C(\beta) r^{\beta}\qquad\mbox{$B_r\subset B_{1/2}$}
\ee
using \eqref{5.1} and \eqref{5.2} we have
\be\label{5.3}
\begin{split}
\sup_{B_r}|u(x)-u(0)|&\le \sup_{B_r}\Big ( |u(x)-u_0(x)|+|u_0(x)-u_0(0)|+|u_0(0)-u(0)| \Big ) \\
&\le 2\eps+C(\beta )r^{\beta}
\end{split}
\ee
now select $\beta$ such that, $1>\beta>\alpha$ and $r=r_0>0$ such that 
$$
C(\beta)r_0^{\beta}=\frac{r_0^{\alpha}}{3}
$$
that is 
$$
r_0=\Big (  \frac{1}{3C(\beta)} \Big )^{1/\beta-\alpha}
$$
and select $\eps>0$ (and accordingly $\delta(\eps)>0$) such that 
$$
\eps<\frac{r_0^{\alpha}}{3}
$$
now from \eqref{5.3} and using the assumption that $u(0)=0$, we prove the lemma.
\end{proof}

\begin{lemm}\label{far from FB}
Following the hypothesis of the previous lemma, for all $0<\alpha<1$ there exists $C_0>0$ depending only on $\alpha$ and data of the problem such that if
$$
\|\A-\A_0\|_{L^1(B_1)}<\delta
$$
implies 
$$
\sup_{B_{r}}|u(x)|\le C_0 r^{\alpha} \qquad \mbox{for $r<r_0$}.
$$
($r_0$ and $\delta$ are as in the Lemma \ref{far from FB 0}.)
\end{lemm}
\begin{proof}
Let us first show that for all $k\in \N$
$$
 \sup_{B_{r_0^{k}}}|u(x)|\le r_0^{k\alpha }.
$$
We already know from the Lemma \ref{far from FB 0} that the above claim is true for $k=1$. Suppose it is true up to a value $k_0\in \N$. We claim that it is also true for $k_0+1$. 

Define $v(y)=\frac{1}{r_0^{k_0\alpha}}u(r_0^{k_0}y)$ for $y\in B_1=B_1(0)$.
We have 
$$
\nabla v(y)=\frac{1}{r_0^{k_0(\alpha-1)}}\nabla u(r_0^ky)
$$
by change of variables we have
\be
\begin{split}
\F_{\A}(u;B_{r_0^{k_0}})=&r_0^{Nk_0}\int_{B_1}r_0^{pk_0(\alpha-1)} \A(r_0^{k_0}y,v) |\nabla v(y)|^{p}\,dy- \\
&r_0^{Nk_0}\int_{B_1}r_0^{k_0\alpha}vf(r_0^{k_0}y,v)+\gamma(r_0^{k_0}y,v)\,dy.
\end{split}
\ee
\end{proof}
See that, $v$ is a minimizer of $\tilde \F=\F_{\tilde \A, \tilde f,\tilde \gamma}(v;B_1)$ with 
\[
\begin{split}
&\tilde \A(y,s)=\A(r_0^{k_0}y,s)\\
&\tilde f(y,s)=r_0^{(pk_0(1-\alpha)+k_0\alpha)}f(r_0^{k_0}y,s)\\
&\tilde \gamma (y,s)=r_0^{pk_0(1-\alpha)} \gamma (r_0^{k_0}y,s)
\end{split}
\]
moreover, 
$$
\|\tilde f\|_{L^{N}(B_1)}=r_0^{k_0(1-\alpha)(p-1)}\|f\|_{L^N(B_{r^k})}\le \|f\|_{L^N(B_1)}.
$$
Since $\sup _{B_{1}}|v|= r_0^{-k_0 \alpha} \sup_{B_{r_0^{k_0}}}|u|\le 1$, $ \tilde \F$ satisfies the hypothesis of the Lemma \ref{far from FB 0} and since $0\in \partial \{v>0\}$ we have
$$
\sup_{B_{r_0}}|v|\le r_0^{\alpha}
$$
substituting $u$ from the definition of $v$
$$
\sup_{B_{r_0^{k_0+1}}}|u|\le r_0^{(k_0+1)\alpha}.
$$
This proves the claim. Now we prove the lemma using a classical iteration argument. Let $0<r<r_0$ and select $k$ such that $r_0^{k+1}<r<r_0^k$, and we see that 
$$
\sup_{B_r}|u|\le \sup_{B_{r_0^{k}}}|u|\le r_0^{k\alpha}=r_0^{(k+1)\alpha}\frac{1}{r_0^{\alpha}}\le \frac{1}{r_0^{\alpha}}r^{\alpha}.
$$
This concludes the proof.

\begin{rema}\label{rem3.4}

We have obtained a local asymptotic $C^{0,1^-}$ regularity estimates on $u$ at points on the free boundary. If $\A_{\pm}(x)$ are separately continuous and $f_{\pm}\in L^N$, minimizer $u$ satisfy the following Euler Lagrangian equation in positive and negative phases
$$
\begin{cases}
-\dive(\A_+(x)|\nabla u|^{p-2}\nabla u)=f_+\qquad \mbox{in $\{u>0\}$}\\
-\dive(\A_-(x)|\nabla u|^{p-2}\nabla u)=f_-\qquad \mbox{in $\{u<0\}$}.
\end{cases}
$$
Since $p<N$, from \cite[Theorem 4.2]{EVT11} we have local $C^{0,1^-}$ regularity estimates on $u$ away from the free boundary. However, these regularity estimates deteriorate as we move closer to the free boundary and therefore we cannot yet conclude that local asymptotic $C^{0,1^-}$ regularity estimates holds in the entire domain under consideration. In order to prove it, we will proceed using our information on how those estimates obtained in \cite{EVT11} deteriorate near the free boundary and make use of the non-homogenous Moser-Harnack inequality with some geometric arguments.
\end{rema}

\begin{lemm}\label{main result}
Let $\A_{\pm}\in C(B_1)$ and with the same hypothesis as in the previous theorem, for all $0<\alpha<1$ there exists $\delta>0$ depending on $\alpha$ and data of the problem such that if 
$$
\|\A-\A_0\|_{L^1(B_1)}< \delta
$$
we have  
$$
u\in C^{\alpha}(B_{1/2}).
$$
\end{lemm}
\begin{proof}
As mentioned in the previous remark, we only need to show that the constant of $\alpha$ H\"older continuity does not blow up as we move closer to the free boundary. This information along with the local estimate in \cite[Theorem 4.2]{EVT11}, will prove our claim.

Let $x_0\in \{u>0\}\cap B_{1/2}$ be very close to the free boundary. From \cite[Theorem 1]{serrin63}, we know that $u$ satisfy the  non-homogenous Moser-Harnack inequality for $r<d/4$ ($d=\dist(x_0,\partial \{u>0\})$)
\be\label{harnack}
\sup_{B_r(x_0)} u\le C \left (  \inf_{B_{r/2}}u+r\|f\|_{L^N(B_r(x_0))} \right ).
\ee
Also, we know from [\cite{EVT11}, Theorem 4.2] and Campanato's theorem that $u$ satisfy interior $C^{0,1^-}$ estimates in $\{u>0\}$ which, for a given $0<\alpha<1$, can be written as follows
\be\label{schauder}
\|u\|_{C^{\alpha}(B_r(x_0))} \le  \|u\|_{C^{\alpha}(B_{d/2}(x_0))}\le \frac{C}{d^{\alpha}}\|u\|_{L^{\infty}(B_{2d/3})}
\ee

using \eqref{harnack} in \eqref{schauder} by putting appropriate value of $r>0$, we can get
\be\label{control 0}
\|u\|_{C^{\alpha}(B_{r}(x_0))}\le \frac{C_1}{d^{\alpha}} \left ( u(x_0)+d\|f\|_{L^N(B_1)}  \right ).
\ee
And now we make use of the Lemma \ref{far from FB}. Let $y_0\in \{u>0\}$ be a point such that $d=\dist(x_0,y_0)=\dist (x_0,\partial \{u>0\})$. From the Lemma \ref{far from FB} we have
\be\label{holder on FB}
\sup_{B_{2d}(y_0)}|u|\le C_2 d^{\alpha}.
\ee
Observe that $x_0\in B_{2d}(y_0)$ and then we combine \eqref{control 0} and \eqref{holder on FB} to get
$$
\|u\|_{C^{\alpha}(B_{r}(x_0))}\le C_1C_2+C_1d^{1-\alpha}\|f\|_{L^N(B_1)}
$$
that is 
$$
\|u\|_{C^{\alpha}(B_{r}(x_0))}\le C_3+C_12^{1-\alpha}\|f\|_{L^N(B_1)}.
$$
From the Remark \ref{rem3.4}, we already know that $u\in C_{\mathrm{loc}}^{0,\alpha}(\{u>0\}\cap B_{1/2})$ and $u\in C_{\mathrm{loc}}^{0,\alpha}(\{u<0\}\cap B_{1/2})$. Hence, the asymptotic $C^{0,1^-}$ regularity estimates on $u$ in $B_{1/2}$ is proved.
\end{proof}

\begin{theo}[Regularity in a continuous medium]\label{continuous medium}
If $u\in W_{\phi}^{1,p}(\O)$ is minimizer of $\F(\cdot,\O)$, $\A_{+}=\A_{-}=\A\in C(\O)$ and $f_{\pm}\in L^N(\O)$. Then $u\in C^{0,1^-}_{\mathrm{loc}}(\O)$.
\end{theo}
\begin{proof}
The proof of this theorem follows simply by rescaling argument. Let $\O'\Subset \O$, set $d=\dist (\O^{c},\O')$ and $\O^{''}=\{x\in \O \mid \dist(x,\O')<d/2\}$. From previous lemmas, we know that $u$ is uniformly bounded in $\overline {\O^{''}}$ and moreover $\A$ is uniformly continuous in $\overline{\O^{''}}$.

Given $0<\alpha<1$, choose corresponding $\eps>0$ from the above lemmas and $\delta<d/2$ such that $|\A(x)-\A(y)|<\eps$ when $|x-y|<\delta$, ($x,y\in \O'$). Let $x_0\in \O'$ and fix $\A_0=\A(x_0)$, rescale $u_\delta(y)=u(x_0+\delta y)$ for $y\in B_1$. We can apply the Lemma \ref{main result} to $u_\delta$ and obtain $u_{\delta}\in C^{0,\alpha}(B_{1/2})$. Thus we can conclude that $u\in C^{0,\alpha}(B_{\delta/2}(x_0))$. Covering $\O'$ with balls of radius $\delta$, we can prove the theorem.
\end{proof}

\section{Asymptotic regularity estimates}\label{section4}

Now, we will use the same strategy as in the previous section and show that the regularity of a minimizer of $\F$ with $\A_{\pm}\in C(\O)$ and $f_{\pm}\in L^N(\O)$ tends asymptitically to locally Lipschitz regularity as the $L^p(\O)$ distance of $\A_+$ and $\A_-$ becomes smaller. That is, given any $0<\alpha<1$ there is a distance $\delta>0$ such that if $\|\A_+-\A_-\|_{L^1(\O)}$ is smaller than $\delta$, then $u\in C_{\mathrm{loc}}^{0,\alpha}(\O)$.

The result of the Theorem \ref{continuous medium} is the limit case when distance between $\A_+$ and $\A_-$ is zero. In fact, instead of considering $\F_0$ (defined in the \eqref{F0}, Lemma \ref{convergence}), we will consider the functional in the hypothesis of the Theorem \ref{continuous medium} as a tangential free boundary problem and recover regularity in converging solutions. Note that Lipschitz regularity for minimizer does not hold true even when $\A_+=\A_-=\mbox{constant}$ (see \cite{LTQ15}). 

We will be assuming that $\A_+$ and $\A_-$ are separately continuous in $\O$ with a modulus of continuity $\omega$

\be\label{modulus of continuity}
|\A_{\pm}(x)-\A_{\pm}(y)|\le \omega(|x-y|).
\ee
\begin{theo}
Let $u\in W^{1,p}_{\phi}(\O)$ be a minimizer of $\F$ with $\A_{\pm}$ satisfying \eqref{modulus of continuity} and $f_{\pm}\in L^N(\O)$. Then given $\O'\Subset \O$, for all $0<\alpha<1$ there exists a $\delta >0$ depending on $\alpha,\O'$ and data of the problem such that if 
\be\label{Lpconvergence}
\|\A_{+}-\A_{-}\|_{L^{1}(\O)}\le \delta
\ee
then
$$
u\in C^{0,\alpha}(\O').
$$
\end{theo}
\begin{proof}
The proof of the theorem is in the same lines as we proceeded in the previous section. We outline a very brief sketch of the proof for the reader.

For the reasons same as that discussed in the beginning of the Section \ref{section3}, we prove results for the unit ball $B_1$ centred at $0$ and assume $\|u\|_{L^{\infty}(B_1)}=1$.

Assume for a sequence $\{\A_{\pm}^k\}\in C(\O)$ satisfying \eqref{modulus of continuity} such that $\|\A_+^k-\A_-^k\|_{L^{1}(B_1)}\rightarrow 0$. Where $u_k$ is minimizer of functional $\F_k=\F_{\A_k}(\cdot;B_1)$. We have defined $\A^k(x,s)$ as 
$$
\A^k{k}(x,s)=\A^k_+(x)\chi_{\{s>0\}}+\A^k_-(x)\chi_{\{s\le 0\}}.
$$
To show that the sequence $\{u_k\}$ uniformly converges to a $C^{0,1^-}$ function in $B_{1/2}$, we argue by contradiction. Let us assume that there exists $\eps_0>0$ such that for every $w\in C^{0,1^-}(B_{1/2})$ we have $\|u_k-w\|_{L^{\infty}(B_{1/2})}>\eps_0$.

Then we can argue in the same way as in the proof of the Lemma \ref{convergence} contradict the claim. 

Since $\A_{\pm}^k$ satisfy \eqref{Lpconvergence} and \eqref{modulus of continuity}, we can apply Egorov's theorem and then Arzela Ascoli theorem to $\{\A^k\}$ in $B_{1/2}$. Thus the sequence $\{\A^k\}$ converges (upto a subsequence) uniformly to a continuous function $\A^*$ satisfying \eqref{modulus of continuity}. Then proceeding as in \eqref{principle part} and \eqref{part2}, we know that the sequence $\{u_k\}$ converge uniformly in $B_{1/2}$ to minimizer  $u^*$ of $\F^*$ defined as 
$$
\F^{*}(v)=\int_{B_1}\A^*(x)|\nabla v|^p-f(x,v)v+\lambda(x,v)\,dx .
$$
 In order to show the regularity estimates on $u$, we will be considering $\F^*$ as the tangential free boundary problem insteal of $\F_0$ as in previous section. 

From the Theorem \ref{continuous medium} we know $u^*\in C^{0,1^-}(B_{1/2})$. Hence we get a contradiction and we show that $u_k\to u^{*}\in C^{0,1^-}(B_{1/2})$ uniformly in $B_{1/2}$. 

Analogous rescaling arguments as in the proof of the Lemma \ref{far from FB} can be used to show the asymptotic $C^{0,1^-}$ estimates on $u$ in $B_{1/2}$ at points on the free boundary. 
Then we can prove the theorem using the Schauder type estimates (see \eqref{schauder}) and the non-homogenous Moser-Harnack (see \eqref{harnack}) inequality with same geometric arguments as in the Theorem \ref{main result}.
\end{proof}

\section{Finite perimeter of the free boundary}\label{section5}

In this section, we will prove that the free boundary of a minimizer of $\F$ is a set of finite perimeter. This result marks the impact of the last term $\gamma_{\pm}(x)$ on the nature of the problem. Heuristically speaking,  this term compensates the transition of phases and thus imposes some flux balance along the free boundary. This in turn forces the free boundary to gain some regularity. 

The technique we will be adapting is from geometric measure theory. One can refer to \cite{db12}. Also refer to \cite{ButHar18} to see an application of the same technique in a shape optimization problem. The main idea is very simple, the information related to the perimeter of the free boundary is present in the integral $\int_{|u|>0}|\nabla u|\,dx$, and the link between this integral and the perimeter is expressed through the co-area formulae. 

We will be assuming the Dirichlet boundary condition and that the terms $\gamma_{\pm}$ are strictly ordered. That is 
\be\label{ordering}
0<c<\gamma_{+}(x)-\gamma_-(x).
\ee

\begin{theo}\label{finite perimeter}
Given a minimizer $u$ of $\F$ with $\gamma$ satisfying \eqref{ordering} and $\phi=0$ on $\partial \O$, the reduced boundary $\partial^* \{u>0\}$ is a set of finite perimeter.
\end{theo}
\begin{proof}
For $\eps>0$ define the following
$$
u_{\eps}=(u-\eps)^+-(u+\eps)^-\mbox{  and   } A_{\eps}=\{0<u\le \eps\}\cap \O.
$$
Note that
$$u_\eps=\begin{cases}
u-\eps&\hbox{if }u>\eps\\
u+\eps&\hbox{if }u<-\eps\\
0&\hbox{if }|u|\le\eps
\end{cases}
\qquad\hbox{and}\qquad
\nabla u_\eps=\begin{cases}
\nabla u&\hbox{a.e. on }\{|u|>\eps\}\\
0&\hbox{a.e. on }\{|u|\le\eps\}.
\end{cases}$$

From the minimility of $u$ for the functional $\F$ we have
$$
\F(u)\le \F(u_{\eps})
$$
therefore
\[
\begin{split}
\int_{\O}\A(x,u)|\nabla u|^p-\A(x,u_{\eps})|\nabla u_{\eps}|^p\,dx +\int_{\O} \gamma(x,u)-\gamma(x,u_{\eps})\,dx & \le C(f,\O,N)\eps\\
\Rightarrow \int_{\{-\eps<|u|\le \eps \}}\A(x,u)|\nabla u|^p\,dx +\int_{\{0<u\le \eps\}}(\gamma_+(x)-\gamma_-(x))\,dx&\le C(f,\O,N)\eps.
\end{split}
\]

Using the hypothesis \eqref{ellipticity} and \eqref{ordering} in above inequation, we obtain the following
$$
\int_{A_{\eps}}\lambda |\nabla u|^p \,dx +c |A_{\eps}|\,dx\le C(f,\O,N)\eps
$$

this implies that $\int_{A_{\eps}}|\nabla u|^p\,dx \le C\eps$ and $|A_{\eps}|\le C\eps$, from the H\"older inequality, we have
$$
\int_{\{0<u<\eps\}}|\nabla u|\,dx \le \int_{A_{\eps}}|\nabla u|\,dx\le \left ( \int_{A_{\eps}}|\nabla u |^p\,dx\right )^{1/p}|A_{\eps}|^{1/p'}\le C\eps
$$
where $C=C(f,\lambda,c,\O,N)$.
Now, we use the coarea formula and we deduce
$$\int_0^\eps\HH^{N-1}\big(\partial^*\{u>t\}\big)\,dt\le C\eps.$$
Thus there exists a sequence $\delta_n\to0$ such that
$$\HH^{N-1}\big(\partial^*\{u>\delta_n\}\big)\le C\qquad\hbox{for every }n.$$
Since $\{u>\delta_n\}\to \{u>0\}$ in $L^1(\O)$ and perimeter is lower semicontinuous with respect to $L^1-$convergence of sets (see \cite{AFP00}), we finally imply that
$$\HH^{N-1}\big(\partial^*\{u>0\}\big)\le C$$
as required.
\end{proof}

\begin{rema}
After adding an assumption that $\gamma(x,0)=0, \gamma(x,s)>c>0\;(s\neq 0)$ and following the lines of the proof of the Theorem \ref{finite perimeter}, one can show that the reduced boundary $\partial ^*\{u<0\}$ is also a set of finite perimeter. In this case, the ordering condition \eqref{ordering} in the hypothesis of the Theorem \ref{finite perimeter} can be dropped.
\end{rema}

\begin{rema}\label{general boundary}
We can prove a local version of the Theorem \ref{finite perimeter} for a general boundary condition, using the same ideas. For this, we consider a ball $B_r$, such that $B_{2r}$ is contained inside the domain $\O$ and we define the test function $\tilde{u_\eps}$ as
$$
\tilde {u_{\eps}}  =\eta u+ (1-\eta)u_{\eps}
$$
where $u_{\eps}$ is defined same as in the proof of Theorem \ref{finite perimeter} and $\eta$ is a smooth function such that 
$$
\eta(x)= \begin{cases}
0 \;\; \mbox{if $x\in B_r$}\\
1 \;\; \mbox{if $x\in \O\setminus B_{2r}$}.
\end{cases}
$$
If we make use of $\tilde{u_{\eps}}$ instead of $u_{\eps}$ as a test function in the proof of the Theorem \ref{finite perimeter}, we can show that the reduced free boundary is a set of finite perimeter inside the ball $B_r$.  More elaborate discussion can be found in the Section 5.2 of \cite{velichkovnotes}, for the case of one phase free boundary problem.

\end{rema}

\end{document}